\documentclass[12pt,draft,amsmath,amscd,amsbsy,amssymb]{amsart}
 \relax

\chardef\bslash=`\\ % p.  424, TeXbook %\newcommand{\ntt}{\seriesm\shape n\tt}

\makeatletter \def\verbatim{\interlinepenalty\@M \@verbatim
	\leftskip\@totalleftmargin\advance\leftskip1pc
	\frenchspacing\@vobeyspaces \@xverbatim} \makeatother \hfuzz1pc

\makeatletter \def\dgt@k{\dg@DX=-3 \dg@DY=2 \dg@SIZE=3}
\makeatother

\makeatletter \def\dgt@kk{\dg@DX=3 \dg@DY=-1 \dg@SIZE=3}% \makeatother

\theoremstyle{plain} \newtheorem{theorem}{Theorem}[section]
\newtheorem{lemma}[theorem]{Lemma}
\usepackage[all]{xy}
\begin{document}

%===============================================================================
%   Заголовок --- назва, анотація, і т. ін.
%===============================================================================

\author{Nataliya Mazurenko}
\email{mnatali@ukr.net}
\address{Department of Mathematics and Computer Science, Vasyl Stefanyk Precarpathian National University,
	57 Shevchenka str., Ivano-Frankivsk, 76025, Ukraine}

\author{Khrystyna Sukhorukova}
\email{kristinsukhorukova@gmail.com}
\address{Ivan Franko Lviv National University, 1 Universytetska str., 79000, Lviv, Ukraine}

\author{Mykhailo Zarichnyi}
\email{zarichnyi@yahoo.com}
\address{Institute of Mathematics, College of Natural Sciences, University of Rzesz\'ow, Pigonia 1, 35-310 Rzesz\'ow, Poland}
%\orcid{Ваш номер orcid}

\title[Invariant idempotent $\ast$-measures]{Invariant idempotent $\ast$-measures}

%\abstract{ukrainian}{Анотація українською мовою.}
\begin{abstract} The notion of $\ast$-idempotent measure is a modification of the notion of idempotent measure defined for every triangular norm $\ast$. We prove existence and uniqueness of invariant $\ast$-idempotent measures for iterated function systems on compact metric spaces.
	\end{abstract}
%\abstract{russian}{Анотация на русском языке.}

%\shortAuthorsList{N. Mazurenko}
%\udc{515.12}
%\thanks{Дослідження виконано за підтримки ...}

\maketitle

%===============================================================================
%   Текст статті
%===============================================================================

\section{Introduction}

Idempotent mathematics is a part of mathematics in which one of the ordinary arithmetic operations in $\mathbb R$ is replaced by an idempotent operation (e.g., maximum). 

The notion of probability measure has its counterparts in idempotent mathematics. The idempotent measures (also called Maslov measures) are introduced in \cite{KM}. The topological and categorical aspects of the theory of idempotent measures are considered in \cite{Z}. Also, in \cite{BZ}  theory of the max-min measures, which can also be regarded as idempotent counterparts of probability measures, is developed.

Some analogous classes of measures, namely the so called *-measures, for any triangular norm $\ast$, are considered in \cite{Su}.

In \cite{MZ} the invariant idempotent measures for iterated function systems are defined and the existence and uniqueness theorem for such measures is proved. The proof is based on functional representation of measures. In \cite{COS} the authors proved the  existence and uniqueness of invariant idempotent measures by using Zaitov's metric  \cite{Za} as well as  modified Bazylevych-Repov\v s-Zarichnyi metric on the space of idempotent measures \cite{BRZ} and applying the Banach contraction principle.

In the present note we provide a simple proof of existence and uniqueness of invariant idempotent measures that works also for the $\ast$-idempotent measures in the sense of \cite{Su}.

\section{Preliminaries}

\subsection{$\ast$-measures}

A triangular norm (t-norm) is a continuous, associative, commutative and monotonic operation on the unit segment $\mathbb I=[0,1]$ for which 1 is a unit.  (see, e.g., \cite{KMP}).

Let $X$ be a compact Hausdorff space. By $C(X,\mathbb I)$ we denote the set of all continuous functions from $X$ to $\mathbb I$. Given $c\in \mathbb I$, we denote by $c_X\in C(X,\mathbb I)$ the constant function with the value $c$. A functional $\mu\colon C(X,\mathbb I)\to\mathbb I$ is called an idempotent $\ast$-measure \cite{Su} if

\begin{enumerate}
	\item $\mu(c_X)=c$;
	\item $\mu (\lambda\ast \varphi)=\lambda\ast \mu (\varphi)$;
	\item $\mu (\varphi\vee \psi)=\mu(\varphi)\vee\mu(\psi)$,
\end{enumerate}
for all $c\in\mathbb I$ and $\varphi,\psi\in C(X,\mathbb I)$. (Hereafter $\vee$ denotes the maximum.)

Let $M^\ast(X)$ denote the set of all $\ast$-measures on $X$.

Given $\mu\in M^\ast(X)$, $\varphi_1,\dots,\varphi_n\in C(X,\mathbb I)$ and $\varepsilon>0$, we define  $$O\langle\mu;\varphi_1,\dots,\varphi_n;\varepsilon\rangle=\{\nu\in M^\ast(X)\mid |\mu(\varphi_i)-\nu(\varphi_i)|<\varepsilon,\ i=1,\dots,n\}.$$

Then the sets of the form $O\langle\mu;\varphi_1,\dots,\varphi_n;\varepsilon\rangle$, where $\mu\in M^\ast(X)$, $\varphi_1,\dots,\varphi_n\in C(X,\mathbb I)$, $n\in C(X,\mathbb N$ and $\varepsilon>0$, comprise a base of the weak* topology on $M^\ast(X)$.

Given a continuous map $f\colon X\to Y$ of compact Hausdorff space, and $\mu \in M^\ast(X)$, define $M^\ast(f)(\mu):  C(X,\mathbb I)\to\mathbb I$ as follows: $$M^\ast(f)(\mu)(\varphi)=\mu(\varphi f).$$

Actually, $M^\ast$ is a functor from the category $\mathbf{Comp}$ of compact Hausdorff spaces and continuous maps to itself.

\subsection{Hyperspaces} Let $X$ be a compact Hausdorff space. By $\exp X$ we denote the set of all nonempty compact subsets in $X$. Given $U_1,\dots,U_n\subset X$, let $$\langle U_1,\dots,U_n\rangle=\{A\in \exp X\mid A\subset \cup_{i=1}^nU_i,\ A\cap U_i\neq\emptyset,\ i=1,\dots,n\}.$$ The family $$\{\langle U_1,\dots,U_n\rangle\mid U_1,\dots,U_n\text{ are open},\ n\in\mathbb N\}$$ is known to be a base of the {\em Vietoris topology} on the set $\exp X$.
The obtained topological space is called the {\em hyperspace} of $X$.

If $(X,d)$ is a compact metric space, then the set $\exp X$ is endowed with the 
Hausdorff metric $d_H$,
$$d_H(A,B)=\max \left\{\sup _{x\in A}\inf _{y\in B}d(x,y),\;\sup _{y\in B}\inf _{x\in A}d(x,y)\right\}.$$

The Hausdorff metric is known to generate the Vietoris topology.

In the following lemma, the sup metric is considered on the product.

\begin{lemma}\label{prod} Let $X,Y$ be compact metric spaces,  $\mathrm{pr}_Y\colon X\times Y\to Y$ be the projections. Let $A,B\in\exp (X\times Y)$. If $\mathrm{pr}_Y(A)=\mathrm{pr}_Y(B)$, then $$d_H(A,B)\le\mathrm{diam}(X).$$
	\end{lemma}

If $f\colon X\to Y$ is a map, then the map $\exp f\colon \exp X\to\exp Y$ is defined as follows: $\exp f(A)=f(A)$, $A\in \exp X$.

Denote by $\bar M(X)$ the set of all $A\in\exp (X\times\mathbb I)$ satisfying the following conditions:
\begin{enumerate}
	\item $A\cap (X\times\{1\})\neq\emptyset$;
	\item $X\times\{0\}\subset A$;
	\item $A$ is saturated, i.e., if $(x,t)\in A$, then $(x,s)\in A$ for every $s\in[0,t]$.
\end{enumerate}

Let $f\colon X\to Y$ be a map. Define the map $\bar M(f)\colon \bar M(X)\to\bar M(Y)$ by the formula:
$$\bar M(f)(A)=\exp(f\times 1_{\mathbb I})(A)\cup (Y\times\{0\}).$$

Actually, $\bar M$ is a functor in the category $\mathbf{Comp}$.

The functors $M^\ast$ and $\bar M$ are known to be isomorphic for all triangular norms $\ast$ (see \cite{Su}). Note that this is an isomorphism of functors, but not of more complicated structures, e.g., monads (see \cite{SZ}). For the sake of reader's convenience, in the sequel we deal with $\bar M$ having a triangular norm $\ast$ in mind.

Given $A\in \bar M(X)$ and $r\in\mathbb I$, define $$r\bar{\ast} A=\{(x,r\ast t)\mid (x,t)\in A\}.$$

\section{Iterated Function Systems}

Let $(X,d)$ be a compact metric space. A map $f\colon X\to X$ is called a contraction if there exists $c\in[0,1)$ such that $d(f(x),f(y))\le cd(x,y)$, for all $x,y\in X$. Here, $c$ is a contraction constant of $f$.

An Iterated Function System is a collection $f_i\colon X\to X$ of contractions.

Let $\lambda_1,\dots,\lambda_n\in[0,1]$ be such that $\vee_{i=1}^n\lambda_i=1$.

Define $\Psi\colon \bar M(X)\to \bar M(X)$ as follows:
$$\Psi(A)=\cup_{i=1}^n\lambda_i\bar{\ast}\bar M(f_i)(A).$$

An element $A\in \bar M(X)$ is called invariant if $\Psi(A)=A$.

\begin{lemma}\label{subset} There exists  $A\in \bar M(X)$ such that $\Psi(A)\subset A$. 
	\end{lemma}

\begin{proof} There exists $B\in \exp X$ such that $f_i(B)\subset B$ for every $i=1,2,\dots,n$ (see  \cite{Hu}, Section 3.1, Theorem (3), Proof of Existence). Then $A=(B\times\mathbb I)\cup (X\times\{0\})$ is as required.

\end{proof}

\begin{theorem} There is a unique invariant element in $\bar M(X)$.
	\end{theorem}

\begin{proof} 
	Let $A$ be as in Lemma \ref{subset}. Then $$A\supset \Psi(A)\supset \Psi^2(A)\supset \Psi^3(A)\supset\dots$$
	and clearly $A_\infty=\cap_{i=1}^\infty  \Psi^i(A)$ is an invariant set.	
	
We are going to prove the uniqueness. Suppose that $A,A'$ are invariant elements.	Let $c>0$ be a contraction constant for all $f_i$, $i=1,\dots,k$. Then, for any natural $n$,
\begin{align*}d_H&(A,A')=d_H(\Psi^n(A),\Psi^n(A'))\\
&\le d_H(\bigcup\{(\lambda_{i_1}\ast\dots\ast\lambda_{i_n})\bar{\ast} \bar M(f_{i_1}\dots f_{i_n})(A)\mid (i_1,\dots,i_n)\in \{1,\dots,k\}^n\},\\
& \bigcup\{(\lambda_{i_1}\ast\dots\ast\lambda_{i_n})\bar{\ast} \bar M(f_{i_1}\dots f_{i_n})(A')\mid (i_1,\dots,i_n)\in \{1,\dots,k\}^n\})\\
&\le \max\{d_H((\lambda_{i_1}\ast\dots\ast\lambda_{i_n})\bar{\ast} \bar M(f_{i_1}\dots f_{i_n})(A),(\lambda_{i_1}\ast\dots\ast\lambda_{i_n})\bar{\ast} \bar M(f_{i_1}\dots f_{i_n})(A'))\\
&\mid (i_1,\dots,i_n)\in \{1,\dots,k\}^n\}\\
\le& \max\{d_H((\lambda_{i_1}\ast\dots\ast\lambda_{i_n})\bar{\ast} ((f_{i_1}\dots f_{i_n})\times 1_{\mathbb I})(A),(\lambda_{i_1}\ast\dots\ast\lambda_{i_n})\bar{\ast}  ((f_{i_1}\dots f_{i_n})\times 1_{\mathbb I})(A'))\\
&\mid (i_1,\dots,i_n)\in \{1,\dots,k\}^n\} \\
&\le c^n\mathrm{diam}(X)
\end{align*}
by Lemma \ref{prod}. Therefore $A=A'$.
	
	\end{proof}

\section{Remarks}	

Non-finite IFSs are considered in \cite{COS1}. We believe that our result can also be extended over a non-finite (e.g., compact) case.

Some interesting examples of invariant idempotent measures are presented in \cite{COS}. This corresponds to the case of the multiplication t-norm. It is an open question to find meaningful examples of invariant $\ast$-idempotent measures, for another triangular norms $\ast$.

%===============================================================================
%   Література
%===============================================================================

%===============================================================================
%   Інформація про авторів
%===============================================================================

%\printArticleAuthorsInfo{\thearticlesnum}

\end{document}